\documentclass[a4paper,11pt]{article}  
\usepackage{srcltx}
\usepackage{xcolor}
\usepackage{graphicx}
\usepackage{amsmath}
\usepackage{amssymb}
\usepackage{enumitem}  
\usepackage{theorem}
\usepackage{euscript}
\usepackage{epic,eepic}
\usepackage{pstricks}
\definecolor{labelkey}{rgb}{0,0.08,0.45}
\definecolor{refkey}{rgb}{0,0.6,0.0}
\definecolor{Brown}{rgb}{0.45,0.0,0.05}
\definecolor{dgreen}{rgb}{0.00,0.49,0.00}
\definecolor{dblue}{rgb}{0,0.08,0.75}
\title{\sffamily MONOTONE OPERATOR METHODS FOR\\
NASH EQUILIBRIA IN NON-POTENTIAL GAMES
\footnote{Contact author: 
P. L. Combettes, {\ttfamily plc@math.jussieu.fr},
phone: +33 1 4427 6319, fax: +33 1 4427 7200.
This work was supported by the Agence Nationale de la Recherche under
grant ANR-08-BLAN-0294-02.}}
\author{Luis M. Brice\~{n}o-Arias and 
Patrick L. Combettes
\\[5mm]
\small UPMC Universit\'e Paris 06\\
\small Laboratoire Jacques-Louis Lions -- UMR 7598\\
\small 75005 Paris, France \\
\small (lbriceno@math.jussieu.fr, 
plc@math.jussieu.fr)
}

\date{~}
\newcommand{\scal}[2]{{\left\langle{{#1}\mid{#2}}\right\rangle}}
\newcommand{\pscal}[2]{\langle\langle{#1}\mid{#2}\rangle\rangle} 
\newcommand{\Menge}[2]{\Bigg\{{#1}~\bigg |~{#2}\Bigg\}} 
\newcommand{\menge}[2]{\big\{{#1}~\big |~{#2}\big\}} 
\newcommand{\minimize}[2]{\ensuremath{\underset{\substack{{#1}}}%
{\mathrm{minimize}}\;\;#2 }}
\newcommand{\argmin}[2]{\ensuremath{\underset{\substack{{#1}}}%
{\mathrm{argmin}}\;\;#2 }}
\newcommand{\Argmin}[2]{\ensuremath{\underset{\substack{{#1}}}%
{\mathrm{Argmin}}\;\;#2 }}
\newcommand{\Argmax}[2]{\ensuremath{\underset{\substack{{#1}}}%
{\mathrm{Argmax}}\;\;#2 }}

\newcommand{\HH}{\ensuremath{{\mathcal H}}}
\newcommand{\HHH}{\ensuremath{\boldsymbol{\mathcal H}}}

\newcommand{\GG}{\ensuremath{{\mathcal G}}}

\newcommand{\emp}{\ensuremath{{\varnothing}}}
\newcommand{\zer}{\ensuremath{\operatorname{zer}}}
\newcommand{\nabli}[1]{\nabla_{\!{#1}}\,}
\newcommand{\Id}{\ensuremath{\operatorname{Id}}\,}
\newcommand{\RR}{\ensuremath{\mathbb{R}}}

\newcommand{\RP}{\ensuremath{\left[0,+\infty\right[}}

\newcommand{\RPP}{\ensuremath{\left]0,+\infty\right[}}
\newcommand{\RPX}{\ensuremath{\left[0,+\infty\right]}}

\newcommand{\RX}{\ensuremath{\left]-\infty,+\infty\right]}}

\newcommand{\NN}{\ensuremath{\mathbb N}}

\newcommand{\gr}{\ensuremath{\operatorname{gra}}}
\newcommand{\exi}{\ensuremath{\exists\,}}

\newcommand{\pinf}{\ensuremath{{+\infty}}}

\newcommand{\weakly}{\ensuremath{\:\rightharpoonup\:}}
\newcommand{\dom}{\ensuremath{\operatorname{dom}}}
\newcommand{\ran}{\ensuremath{\operatorname{ran}}}
\newcommand{\prox}{\ensuremath{\operatorname{prox}}}

\newtheorem{theorem}{Theorem}[section]

\newtheorem{proposition}[theorem]{Proposition}
\theoremstyle{plain}{\theorembodyfont{\rmfamily}%
}
\theoremstyle{plain}{\theorembodyfont{\rmfamily}%
}
\theoremstyle{plain}{\theorembodyfont{\rmfamily}%
\newtheorem{example}[theorem]{Example}}
\theoremstyle{plain}{\theorembodyfont{\rmfamily}%
}
\theoremstyle{plain}{\theorembodyfont{\rmfamily}%
}
\theoremstyle{plain}{\theorembodyfont{\rmfamily}%
\newtheorem{problem}[theorem]{Problem}}

\numberwithin{equation}{section}
\begin{document}
\maketitle

\begin{abstract}
We observe that a significant class of Nash equilibrium problems
in non-potential games can be associated with monotone inclusion
problems. We propose splitting techniques to solve such problems 
and establish their convergence.
Applications to generalized Nash equilibria, 
zero-sum games, and cyclic proximity problems are demonstrated.
\end{abstract}


\newpage
\section{Problem statement}
\label{sec:1}
Consider a game with $m\geq2$ players indexed by 
$i\in\{1,\ldots,m\}$.
The strategy $x_i$ of the $i$th player lies in a real Hilbert space
$\HH_i$ and the problem is to find $x_1\in\HH_1,\ldots,x_m\in\HH_m$ 
such that
\begin{equation}
\label{e:BR210}
(\forall i\in\{1,\ldots,m\})\:\:\: x_i\in\Argmin{x\in\HH_i}
{\!\!\boldsymbol{f}(x_1,\ldots,x_{i-1},x,x_{i+1},\ldots,x_m)+
\boldsymbol{g}_i(x_1,\ldots,x_{i-1},x,x_{i+1},\ldots,x_m)},
\end{equation}
where $(\boldsymbol{g}_i)_{1\leq i\leq m}$ represents the individual 
penalty of player $i$ depending on the strategies of all players and 
$\boldsymbol{f}$ is a convex penalty which is common to all 
players and models the collective discomfort of the group.
At this level of generality, no reliable method exists for solving 
\eqref{e:BR210} and some hypotheses are required. In this paper we 
focus on the following setting.

\begin{problem}
\label{prob:main}
Let $m\geq2$ be an integer and let 
$\boldsymbol{f}\colon\HH_1\oplus\cdots\oplus\HH_m\to\RX$ be a 
proper lower semicontinuous convex function. For every 
$i\in\{1,\ldots,m\}$, let 
$\boldsymbol{g}_i\colon\HH_1\oplus\cdots\oplus\HH_m\to\RX$
be such that, for every $(x_1,\ldots,x_m)\in
\HH_1\oplus\cdots\oplus\HH_m$, the function 
$x\mapsto\boldsymbol{g}_i(x_1,\ldots,x_{i-1},x,x_{i+1},\ldots,x_m)$
is convex and differentiable on $\HH_i$, and denote by 
$\nabli{i}\boldsymbol{g}_i(x_1,\ldots,x_m)$ its derivative 
at $x_i$. Moreover, 
\begin{multline}
\label{e:Bmon}
\big(\forall (x_1,\ldots,x_m)\in\HH_1\oplus\cdots\oplus\HH_m\big)
\big(\forall (y_1,\ldots,y_m)\in\HH_1\oplus\cdots\oplus\HH_m\big)\\
\sum_{i=1}^m\scal{\nabli{i}\boldsymbol{g}_i
(x_1,\ldots,x_m)-\nabli{i}\boldsymbol{g}_i
(y_1,\ldots,y_m)}{x_i-y_i}\geq0.
\end{multline} 
The problem is to find $x_1\in\HH_1,\ldots,x_m\in\HH_m$
such that
\begin{equation}
\label{e:BR21}
\begin{cases}
\:x_1&\hskip -.3cm
\in\Argmin{x\in\HH_1}{\boldsymbol{f}(x,x_2,\ldots,x_m)
+\boldsymbol{g}_1(x,x_2,\ldots,x_m)}\\
&\hskip -.25cm\vdots\\
x_m&\hskip -.3cm\in\Argmin{x\in\HH_m}{ 
\boldsymbol{f}(x_1,\ldots, x_{m-1},x)
+\boldsymbol{g}_m(x_1,\ldots,x_{m-1},x)}.
\end{cases}
\end{equation}
\end{problem}

In the special case when, for every $i\in\{1,\ldots,m\}$, 
$\boldsymbol{g}_i=\boldsymbol{g}$, Problem~\ref{prob:main} amounts 
to finding a Nash equilibrium of a potential game, i.e., a game in 
which the penalty of player $i$ can be represented by a common 
potential $\boldsymbol{f}+\boldsymbol{g}$ \cite{Mond96}. 
Hence, Nash equilibria can be found by solving
\begin{equation}
\label{e:potent}
\minimize{x_1\in\HH_1,\ldots,x_m\in\HH_m}
{\boldsymbol{f}(x_1,\ldots,x_m)+
\boldsymbol{g}(x_1,\ldots,x_m)}.
\end{equation}
Thus, the problem reduces to the minimization of the sum of two 
convex
functions on the Hilbert space $\HH_1\oplus\cdots\oplus\HH_m$ and
various methods are available to tackle it under suitable assumptions
(see for instance \cite[Chapter~27]{Livre1}). In this paper we 
address the more challenging non-potential setting, in which the 
functions $(\boldsymbol{g}_i)_{1\leq i\leq m}$ need not be 
identical nor convex, but they must satisfy \eqref{e:Bmon}. Let us 
note that \eqref{e:Bmon} actually implies the convexity of 
$\boldsymbol{g}_i$ with respect to its $i$th variable. 

Our methodology consists in using monotone operator splitting 
techniques for solving an auxiliary monotone inclusion, the 
solutions of which are Nash equilibria of Problem~\ref{prob:main}.
In Section~\ref{sec:2} we review the notation and background
material needed subsequently. In Section~\ref{sec:3} we introduce
the auxiliary monotone inclusion problem and provide conditions 
ensuring the existence of solutions to the auxiliary 
problem. We also propose two methods for solving 
Problem~\ref{prob:main} and establish their convergence. 
Finally, in Section~\ref{sec:4} the proposed methods are applied to 
the construction of generalized Nash equilibria, to zero-sum games, 
and to cyclic proximation problems.

\section{Notation and background}
\label{sec:2}
Throughout this paper, $\HH$, $\GG$, and $(\HH_i)_{1\leq i\leq m}$
are real Hilbert spaces. For convenience, their scalar products are 
all denoted by $\scal{\cdot}{\cdot}$ and the associated norms by 
$\|\cdot\|$. Let $A\colon\HH\to 2^{\HH}$ be a set-valued operator.
The domain of $A$ is $\dom A=\menge{x\in\HH}{Ax\neq\emp}$,
the set of zeros of $A$ is $\zer A=\menge{x\in\HH}{0\in Ax}$, the 
graph of $A$ is $\gr A=\menge{(x,u)\in\HH\times\HH}{u\in Ax}$, the 
range of $A$ is $\ran A=\menge{u\in\HH}{(\exi x\in\HH)\;u\in Ax}$, 
the inverse of $A$ is the set-valued operator 
$A^{-1}\colon\HH\to 2^{\HH}\colon u\mapsto\menge{x\in\HH}{u\in Ax}$, 
and the resolvent of $A$ is $J_A=(\Id+A)^{-1}$. In addition, 
$A$ is monotone if 
\begin{equation}
(\forall (x,y)\in\HH\times\HH)(\forall (u,v)\in Ax\times Ay)\quad
\scal{x-y}{u-v}\geq 0
\end{equation}
and it is maximally monotone if, furthermore, every monotone
operator $B\colon\HH\to 2^{\HH}$ such that $\gr A\subset\gr B$
coincides with $A$.

We denote by $\Gamma_0(\HH)$ the class of lower semicontinuous convex 
functions $\varphi\colon\HH\to\RX$ which are proper in the sense that
$\dom\varphi=\menge{x\in\HH}{\varphi(x)<\pinf}\neq\emp$. 
Let $\varphi\in\Gamma_0(\HH)$. The proximity operator 
of $\varphi$ is 
\begin{equation}
\label{e:prox}
\prox_{\varphi}\colon\HH\to\HH\colon x\mapsto\argmin{y\in\HH}
{\varphi(y)+\frac12\|x-y\|^2}
\end{equation}
and the subdifferential of $\varphi$ is the maximally monotone 
operator
\begin{equation}
\label{e:subdiff}
\partial\varphi\colon\HH\to 2^{\HH}\colon x\mapsto
\menge{u\in\HH}{(\forall y\in\HH)\;\;\scal{y-x}{u}+\varphi(x)
\leq \varphi(y)}.
\end{equation}
We have
\begin{equation}
\label{e:fermat}
\Argmin{x\in\HH}{\varphi(x)}=\zer\partial\varphi\quad\text{and}\quad 
\prox_{\varphi}=J_{\partial\varphi}.
\end{equation}
Let $\beta\in\RPP$. An operator $T\colon\HH\to\HH$ is 
$\beta$-cocoercive (or $\beta T$ is firmly nonexpansive) if 
\begin{equation}
\label{e:coco}
(\forall x\in\HH)(\forall y\in\HH)\quad\scal{x-y}{Tx-Ty}
\geq\beta\|Tx-Ty\|^2,
\end{equation}
which implies that it is monotone and $\beta^{-1}$--Lipschitzian.
Let $C$ be a nonempty convex subset of $\HH$. 
The indicator function of $C$ is
\begin{equation}
\label{e:iota}
\iota_C\colon\HH\to\RX\colon x\mapsto
\begin{cases}
0,&\text{if}\;\;x\in C;\\
\pinf,&\text{if}\;\;x\notin C
\end{cases}
\end{equation}
and $\partial \iota_C=N_C$ is the normal cone
operator of $C$, i.e.,
\begin{equation}
N_C\colon\HH\to 2^{\HH}\colon x\mapsto
\begin{cases}
\menge{u\in\HH}{(\forall y\in C)\:\:\scal{y-x}{u}\leq0},
&\text{if }x\in C;\\
\emp,&\text{otherwise}.
\end{cases}
\end{equation}
If $C$ is closed, for every $x\in\HH$, 
there exists a unique point $P_Cx\in C$ such that 
$\|x-P_Cx\|=\inf_{y\in C}\|x-y\|$; 
$P_Cx$ is called the projection of $x$ onto $C$ and we have 
$P_C=\prox_{\iota_C}$.
In addition, the symbols $\weakly$ and $\to$ denote respectively 
weak and strong convergence. For a detailed account of the tools 
described above, see \cite{Livre1}.

\section{Model, algorithms, and convergence}
\label{sec:3}
We investigate an auxiliary monotone inclusion problem the solutions 
of which are Nash equilibria of Problem~\ref{prob:main} and propose 
two splitting methods to solve it. Both involve the proximity 
operator $\prox_{\boldsymbol{f}}$, which can be computed explicitly 
in several instances \cite{Livre1,Jmiv1}.
We henceforth denote by $\HHH$ the direct sum of the Hilbert 
spaces $(\HH_i)_{1\leq i\leq m}$, i.e., the product space
$\HH_1\times\cdots\times\HH_m$ equipped with the 
scalar product
\begin{equation}
\pscal{\cdot}{\cdot}\colon\big((x_i)_{1\leq i\leq m},
(y_i)_{1\leq i\leq m}\big)\mapsto\sum_{i=1}^m\scal{x_i}{y_i}.
\end{equation}
We denote the associated norm by $|||\cdot|||$, a generic 
element of $\HHH$ by $\boldsymbol{x}=(x_i)_{1\leq i\leq m}$,
and the identity operator on $\HHH$ by $\boldsymbol{\Id}$.

\subsection{A monotone inclusion model}

With the notation and hypotheses of Problem~\ref{prob:main}, 
let us set
\begin{equation}
\label{e:B}
\boldsymbol{A}=\partial\boldsymbol{f}\quad\text{and}\quad
\boldsymbol{B}\colon\HHH\to\HHH\colon
\boldsymbol{x}\mapsto\big(\nabli{1}\boldsymbol{g}_1
(\boldsymbol{x}),\ldots,\nabli{m}\boldsymbol{g}_m
(\boldsymbol{x})\big).
\end{equation}
We consider the inclusion problem
\begin{equation}
\label{e:probaux}
\text{find}\quad\boldsymbol{x}\in\zer(\boldsymbol{A}+\boldsymbol{B}).
\end{equation}
Since $\boldsymbol{f}\in\Gamma_0(\HHH)$,
$\boldsymbol{A}$ is maximally monotone. On the other hand,
it follows from \eqref{e:Bmon} that $\boldsymbol{B}$ is monotone.
The following result establishes a connection between the monotone 
inclusion problem \eqref{e:probaux} and Problem~\ref{prob:main}.

\begin{proposition}
\label{p:aux}
Using the notation and hypotheses of Problem~\ref{prob:main},
let $\boldsymbol{A}$ and $\boldsymbol{B}$ be as in \eqref{e:B}. 
Then every point in $\zer(\boldsymbol{A}+\boldsymbol{B})$ is a 
solution to Problem~\ref{prob:main}.
\end{proposition}
\begin{proof}
Suppose that $\zer(\boldsymbol{A}+\boldsymbol{B})\neq\emp$ and
let $(x_1,\ldots,x_m)\in\HHH$. Then
\cite[Proposition~16.6]{Livre1} asserts that
\begin{equation}
\label{e:auxil}
\boldsymbol{A}(x_1,\ldots,x_m)\subset\partial
\big(\boldsymbol{f}(\cdot,x_2,\ldots,x_m)\big)(x_1)\times\cdots
\times\partial\big(\boldsymbol{f}(x_1,\ldots,x_{m-1},
\cdot)\big)(x_m).
\end{equation}
Hence, since $\dom\boldsymbol{g}_1(\cdot,x_2,\ldots,x_m)=\HH_1$,
\ldots, $\dom\boldsymbol{g}_m(x_1,\ldots,x_{m-1},\cdot)=\HH_m$, we 
derive from \eqref{e:B}, \eqref{e:fermat}, and 
\cite[Corollary~16.38(iii)]{Livre1} that
\begin{align}
\label{e:auxil3}
(x_1,\ldots,x_m)\in\zer(\boldsymbol{A}+\boldsymbol{B})
\quad&\Leftrightarrow\quad
-\boldsymbol{B}(x_1,\ldots,x_m)\in\boldsymbol{A}(x_1,\ldots,x_m)
\nonumber\\
&\Rightarrow\quad
\begin{cases}
\:\:-\nabli{1}\boldsymbol{g}_1(x_1,\ldots,x_m)&\hskip -.3cm\in\,
\partial\big(\boldsymbol{f}(\cdot,x_2,\ldots,x_m)\big)(x_1)\\
&\hskip -.3cm\:\vdots\\
-\nabli{m}\boldsymbol{g}_m(x_1,\ldots,x_m)&\hskip -.3cm\in\,
\partial\big(\boldsymbol{f}(x_1,\ldots,x_{m-1},\cdot)\big)(x_m)\\
\end{cases}\nonumber\\
&\Leftrightarrow\quad(x_1,\ldots,x_m)\:\:\:\text{solves
Problem~\ref{prob:main}},
\end{align}
which yields the result.
\end{proof}

Proposition~\ref{p:aux} asserts that we can solve 
Problem~\ref{prob:main} 
by solving \eqref{e:probaux}, provided the latter has solutions. 
The following result provides instances in which this property is 
satisfied. First, we need the following definitions (see 
\cite[Chapters~21--24]{Livre1}).

Let $A\colon\HH\to 2^{\HH}$ be monotone. Then $A$ is $3^*$ monotone
if $\dom A\times\ran A\subset\dom F_A$, where 
\begin{equation}
F_A\colon\HH\times\HH\to\RX\colon (x,u)\mapsto\scal{x}{u}-
\inf_{(y,v)\in\gr A}\scal{x-y}{u-v}.
\end{equation}
On the other hand, $A$ is uniformly monotone if there exists 
an increasing function $\phi\colon\RP\to\RPX$ vanishing only at $0$
such that
\begin{equation}
\label{e:2011-05-24a}
\big(\forall (x,y)\in\HH\times\HH\big)\big(\forall(u,v)\in 
Ax\times Ay\big)\quad
\scal{x-y}{u-v}\geq\phi(\|x-y\|). 
\end{equation}
A function $\varphi\in\Gamma_0(\HH)$ is uniformly convex if there 
exists an increasing function $\phi\colon\RP\to\RPX$ vanishing only 
at $0$ such that
\begin{multline}
\label{e:2011-05-24b}
(\forall (x,y)\in\dom\varphi\times\dom\varphi)
(\forall\alpha\in\left]0,1\right[)\\
\varphi(\alpha x+(1-\alpha)y)+\alpha(1-\alpha)\phi(\|x-y\|)\leq
\alpha\varphi(x)+(1-\alpha)\varphi(y). 
\end{multline}
The function $\phi$ in \eqref{e:2011-05-24a} and 
\eqref{e:2011-05-24b} is called the modulus
of uniform monotonicity and of uniform convexity, respectively,
and it is said to be supercoercive if
$\lim_{t\to\pinf}\phi(t)/t=\pinf$.

\begin{proposition}
\label{p:ex}
With the notation and hypotheses of Problem~\ref{prob:main}, 
let $\boldsymbol{B}$ be as in \eqref{e:B}.
Suppose that $\boldsymbol{B}$ is maximally monotone and that 
one of the following holds.
\begin{enumerate}
\item
\label{p:exi} 
$\lim_{|||\boldsymbol{x}|||\to\pinf}
\inf|||\partial\boldsymbol{f}(\boldsymbol{x})+\boldsymbol{B}
\boldsymbol{x}|||=\pinf$. 
\item\label{p:exii} $\partial\boldsymbol{f}+\boldsymbol{B}$ is 
uniformly monotone with a supercoercive modulus.
\item
\label{p:exvii-} 
$(\dom\partial\boldsymbol{f})\cap\dom {\boldsymbol B}$ is bounded.
\item
\label{p:exvii}  
$\boldsymbol{f}=\iota_{\boldsymbol{C}}$, where 
$\boldsymbol{C}$ is a nonempty closed convex bounded subset of 
$\HHH$.
\item
\label{p:exviii} 
$\boldsymbol{f}$ is uniformly convex with a supercoercive modulus.
\item
\label{p:exiii} 
$\boldsymbol{B}$ is $3^*$ monotone, and $\partial\boldsymbol{f}$ or 
$\boldsymbol{B}$ is surjective.
\item
\label{p:exiv} 
$\boldsymbol{B}$ is uniformly monotone with a supercoercive modulus.
\item
\label{p:exv} 
$\boldsymbol{B}$ is linear and bounded, there 
exists $\beta\in\RPP$ such that $\boldsymbol{B}$ is 
$\beta$--cocoercive,
and $\partial\boldsymbol{f}$ or $\boldsymbol{B}$ is surjective.
\end{enumerate}
Then $\zer(\partial\boldsymbol{f}+\boldsymbol{B})\neq\emp$. In 
addition, if \ref{p:exii}, \ref{p:exviii}, or \ref{p:exiv} holds, 
$\zer(\partial\boldsymbol{f}+\boldsymbol{B})$ is a singleton.
\end{proposition}
\begin{proof}
First note that, for every $\boldsymbol{x}=(x_i)_{1\leq i\leq m}
\in\HHH$, $\dom\nabli{1}\boldsymbol{g}_1(\cdot,x_2,\ldots,x_m)=\HH_1,
\ldots,$ 
$\dom\nabli{m}\boldsymbol{g}_m(x_1,\ldots,x_{m-1},\cdot)=\HH_m$.
Hence, it follows from \eqref{e:B} that $\dom\boldsymbol{B}=\HHH$ 
and, therefore, from \cite[Corollary~24.4(i)]{Livre1} that
$\partial\boldsymbol{f}+\boldsymbol{B}$ is maximally monotone. 
In addition, it follows from \cite[Example~24.9]{Livre1} that 
$\partial\boldsymbol{f}$ is $3^*$ monotone.

\ref{p:exi}: This follows from \cite[Corollary~21.20]{Livre1}.
\ref{p:exii}: This follows from \cite[Corollary~23.37(i)]{Livre1}.
\ref{p:exvii-}: 
Since $\dom(\partial\boldsymbol{f}+\boldsymbol{B})=
(\dom\partial\boldsymbol{f})\cap\dom {\boldsymbol B}$,
the result follows from \cite[Proposition~23.36(iii)]{Livre1}.
\ref{p:exvii}$\Rightarrow$\ref{p:exvii-}: $\boldsymbol{f}=
\iota_{\boldsymbol{C}}\in\Gamma_0(\HHH)$ and
$\dom\partial\boldsymbol{f}=\boldsymbol{C}$ is bounded.
\ref{p:exviii}$\Rightarrow$\ref{p:exii}: 
It follows from \eqref{e:B} and
\cite[Example~22.3(iii)]{Livre1} that $\partial\boldsymbol{f}$ is 
uniformly monotone. Hence, $\partial\boldsymbol{f}+\boldsymbol{B}$ 
is uniformly monotone.
\ref{p:exiii}: This follows from \cite[Corollary~24.22(ii)]{Livre1}.
\ref{p:exiv}$\Rightarrow$\ref{p:exii}: Clear.
\ref{p:exv}$\Rightarrow$\ref{p:exiii}: This follows from 
\cite[Proposition~24.12]{Livre1}.
\end{proof}

\subsection{Forward-backward-forward algorithm}
\label{ssec:1}
Our first method for solving Problem~\ref{prob:main} derives from
an algorithm proposed in \cite{Siopt3}, which is itself 
a variant of a method proposed in \cite{Tsen00}.
\begin{theorem}
\label{t:main}
In Problem~\ref{prob:main}, suppose that there exist 
$(z_1,\ldots,z_m)\in\HHH$ such that
\begin{equation}
\label{e:exist1}
-\big(\nabli{1}\boldsymbol{g}_1
(z_1,\ldots,z_m),\ldots,\nabli{m}\boldsymbol{g}_m
(z_1,\ldots,z_m)\big)\in\partial\boldsymbol{f}
(z_1,\ldots,z_m)
\end{equation}
and $\chi\in\RPP$ such that
\begin{multline}
\label{e:BLip}
(\forall (x_1,\ldots,x_m)\in\HHH)
(\forall (y_1,\ldots,y_m)\in\HHH)\\
\sum_{i=1}^m\|\nabli{i}\boldsymbol{g}_i
(x_1,\ldots,x_m)-\nabli{i}\boldsymbol{g}_i
(y_1,\ldots,y_m)\|^2\leq\chi^2\sum_{i=1}^m\|x_i-y_i\|^2.
\end{multline}
Let $\varepsilon\in\left]0,1/(\chi+1)\right[$ and 
let $(\gamma_n)_{n\in\NN}$ be a sequence in 
$\left[\varepsilon,(1-\varepsilon)/\chi\right]$. Moreover,
for every $i\in\{1,\ldots,m\}$, let $x_{i,0}\in \HH_i$, and
let $(a_{i,n})_{n\in\NN}$, $(b_{i,n})_{n\in\NN}$, and 
$(c_{i,n})_{n\in\NN}$ be absolutely summable sequences in $\HH_i$. 
Now consider the following routine. 
\begin{align}
\label{e:tseng1}
(\forall n\in\NN)\quad
&\left\lfloor 
\begin{array}{l}
\text{For}\:\:i=1,\ldots,m\\
\lfloor\:y_{i,n}=x_{i,n}-\gamma_n(\nabli{i}\boldsymbol{g}_i
(x_{1,n},\ldots,x_{m,n})
+a_{i,n})\\[2mm]
(p_{1,n},\ldots,p_{m,n})=\prox_{\gamma_n\boldsymbol{f}}
(y_{1,n},\ldots,y_{m,n})+(b_{1,n},\ldots,b_{m,n})\\[2mm]
\text{For}\:\:i=1,\ldots,m\\
\left\lfloor
\begin{array}{l}
q_{i,n}=p_{i,n}-\gamma_n(\nabli{i}\boldsymbol{g}_i
(p_{1,n},\ldots,p_{m,n})+c_{i,n})\\
x_{i,n+1}=x_{i,n}-y_{i,n}+q_{i,n}.
\end{array}
\right.
\end{array}
\right.
\end{align}
Then there exists a solution 
$(\overline{x}_1,\ldots,\overline{x}_m)$ to Problem~\ref{prob:main} 
such that, for every $i\in\{1,\ldots,m\}$,
$x_{i,n}\weakly\overline{x}_i$ and $p_{i,n}\weakly\overline{x}_i$.
\end{theorem}
\begin{proof}
Let $\boldsymbol{A}$ and $\boldsymbol{B}$ be defined as 
\eqref{e:B}. Then \eqref{e:exist1} yields
$\zer(\boldsymbol{A}+\boldsymbol{B})\neq\emp$ and, for every 
$\gamma\in\RPP$, \eqref{e:fermat} yields $J_{\gamma\boldsymbol{A}}
=\prox_{\gamma\boldsymbol{f}}$. 
In addition, we deduce from \eqref{e:Bmon} and 
\eqref{e:BLip} that $\boldsymbol{B}$ is monotone 
and $\chi$--Lipschitzian. Now set
\begin{equation}
(\forall n\in\NN)\quad
\begin{cases}
\boldsymbol{x}_n=(x_{1,n},\ldots,x_{m,n})\\
\boldsymbol{y}_n=(y_{1,n},\ldots,y_{m,n})\\
\boldsymbol{p}_n=(p_{1,n},\ldots,p_{m,n})\\
\boldsymbol{q}_n=(q_{1,n},\ldots,q_{m,n})\\
\end{cases}\quad\text{and}\quad
\begin{cases}
\boldsymbol{a}_n=(a_{1,n},\ldots,a_{m,n})\\
\boldsymbol{b}_n=(b_{1,n},\ldots,b_{m,n})\\
\boldsymbol{c}_n=(c_{1,n},\ldots,c_{m,n}).
\end{cases}
\end{equation}
Then \eqref{e:tseng1} is equivalent to
\begin{align}
\label{e:tseng2}
(\forall n\in\NN)\quad
&\left\lfloor 
\begin{array}{l}
\boldsymbol{y}_{n}=\boldsymbol{x}_{n}
-\gamma_n(\boldsymbol{B}\boldsymbol{x}_{n}
+\boldsymbol{a}_{n})\\
\boldsymbol{p}_{n}=J_{\gamma_n\boldsymbol{A}}
\boldsymbol{y}_{n}+\boldsymbol{b}_{n}\\
\boldsymbol{q}_{n}=\boldsymbol{p}_{n}-
\gamma_n(\boldsymbol{B}\boldsymbol{p}_{n}+\boldsymbol{c}_{n})\\
\boldsymbol{x}_{n+1}=\boldsymbol{x}_{n}-\boldsymbol{y}_{n}
+\boldsymbol{q}_{n}.
\end{array}
\right.
\end{align}
Therefore, the result follows from \cite[Theorem~2.5(ii)]{Siopt3} 
and Proposition~\ref{p:aux}.
\end{proof}

Note that two (forward) gradient steps involving the individual 
penalties $(\boldsymbol{g}_i)_{1\leq i\leq m}$ and one (backward) 
proximal step involving the common penalty $\boldsymbol{f}$ are 
required at each iteration of \eqref{e:tseng1}. 

\subsection{Forward-backward algorithm}
\label{ssec:2}
Our second method for solving Problem~\ref{prob:main} 
is somewhat simpler than \eqref{e:tseng1} but requires stronger
hypotheses on $(\boldsymbol{g}_i)_{1\leq i\leq m}$. This method 
is an application of the forward-backward splitting algorithm 
(see \cite{Sico10,Opti04} and the references therein for 
background). 

\begin{theorem}
\label{t:main0}
In Problem~\ref{prob:main}, suppose that there exist 
$(z_1,\ldots,z_m)\in\HHH$ such that
\begin{equation}
\label{e:exist2}
-\big(\nabli{1}\boldsymbol{g}_1
(z_1,\ldots,z_m),\ldots,\nabli{m}\boldsymbol{g}_m
(z_1,\ldots,z_m)\big)\in\partial\boldsymbol{f}
(z_1,\ldots,z_m)
\end{equation}
and $\chi\in\RPP$ such that
\begin{multline}
\label{e:Bcoco}
(\forall (x_1,\ldots,x_m)\in\HHH)
(\forall (y_1,\ldots,y_m)\in\HHH)\\
\hskip -3cm\sum_{i=1}^m\scal{\nabli{i}\boldsymbol{g}_i
(x_1,\ldots,x_m)-\nabli{i}\boldsymbol{g}_i
(y_1,\ldots,y_m)}{x_i-y_i}\\
\geq\frac{1}{\chi}\sum_{i=1}^m\|\nabli{i}\boldsymbol{g}_i
(x_1,\ldots,x_m)-\nabli{i}\boldsymbol{g}_i
(y_1,\ldots,y_m)\|^2.
\end{multline} 
Let $\varepsilon\in\left]0,2/(\chi+1)\right[$ and
let $(\gamma_n)_{n\in\NN}$ be a sequence in 
$[\varepsilon,(2-\varepsilon)/\chi]$. Moreover, 
for every $i\in\{1,\ldots,m\}$, let 
$x_{i,0}\in\HH_i$, and let $(a_{i,n})_{n\in\NN}$ and 
$(b_{i,n})_{n\in\NN}$ 
be absolutely summable sequences in $\HH_i$.
Now consider the following routine.
\begin{align}
\label{e:FB}
(\forall n\in\NN)\quad
&\left\lfloor 
\begin{array}{l}
\text{For}\:\:i=1,\ldots,m\\
\lfloor\:y_{i,n}=x_{i,n}-\gamma_n(\nabli{i}\boldsymbol{g}_i
(x_{1,n},\ldots,x_{m,n})+a_{i,n})\\[2mm]
(x_{1,n+1},\ldots,x_{m,n+1})=\prox_{\gamma_n\boldsymbol{f}}
(y_{1,n},\ldots,y_{m,n})+(b_{1,n},\ldots,b_{m,n}).
\end{array}
\right.
\end{align}
Then there exists a solution 
$(\overline{x}_1,\ldots,\overline{x}_m)$ to Problem~\ref{prob:main} 
such that, for every $i\in\{1,\ldots,m\}$,
$x_{i,n}\weakly\overline{x}_i$ and  
$\nabli{i}\boldsymbol{g}_i(x_{1,n},\ldots,x_{m,n})\to
\nabli{i}\boldsymbol{g}_i(\overline{x}_1,\ldots,\overline{x}_m)$.
\end{theorem}
\begin{proof}
If we define $\boldsymbol{A}$ and $\boldsymbol{B}$ as in 
\eqref{e:B}, \eqref{e:exist2} is equivalent to $\zer(\boldsymbol{A}+
\boldsymbol{B})\neq\emp$, and it follows from \eqref{e:Bcoco}
that $\boldsymbol{B}$ is $\chi^{-1}$--cocoercive. Moreover,
\eqref{e:FB} can be recast as 
\begin{align}
\label{e:2011-05-24c}
(\forall n\in\NN)\quad
&\left\lfloor 
\begin{array}{l}
\boldsymbol{y}_{n}=\boldsymbol{x}_{n}
-\gamma_n(\boldsymbol{B}\boldsymbol{x}_{n}
+\boldsymbol{a}_{n})\\
\boldsymbol{x}_{n+1}=J_{\gamma_n\boldsymbol{A}}
\boldsymbol{y}_{n}+\boldsymbol{b}_{n}.
\end{array}
\right.
\end{align}
The result hence follows from Proposition~\ref{p:aux} and 
\cite[Theorem~2.8(i)\&(ii)]{Sico10}.
\end{proof}

Theorem~\ref{t:main0} imposes more restrictions on 
$(\boldsymbol{g}_i)_{1\leq i\leq m}$. However, unlike the 
forward-backward-forward 
algorithm used in Section~\ref{ssec:1}, it employs only one 
forward step at each iteration. In addition, 
this method allows for larger gradient steps since the sequence 
$(\gamma_n)_{n\in\NN}$ lies in $\left]0,2/\chi\right[$,
as opposed to $\left]0,1/\chi\right[$ in Theorem~\ref{t:main}.

\section{Applications}
\label{sec:4}
The previous results can be used to solve a wide variety of instances
of Problem~\ref{prob:main}. We discuss three examples.

\subsection{Saddle functions and zero-sum games}
We consider an instance of Problem~\ref{prob:main} with $m=2$ players
whose individual penalties $\boldsymbol{g}_1$ and $\boldsymbol{g}_2$ 
are saddle functions. 
\begin{example}
\label{prob:selle}
Let $\chi\in\RPP$, let $\boldsymbol{f}\in\Gamma_0(\HH_1\oplus\HH_2)$,
and let $\boldsymbol{\mathcal{L}}\colon\HH_1\oplus\HH_2\to\RR$ 
be a differentiable function with a $\chi$--Lipschitzian gradient 
such that, for every $x_1\in\HH_1$, $\boldsymbol{\mathcal {L}}
(x_1,\cdot)$ is concave and, for every $x_2\in\HH_2$, 
$\boldsymbol{\mathcal{L}}(\cdot, x_2)$ is convex. 
The problem is to find $x_1\in\HH_1$ and $x_2\in\HH_2$
such that
\begin{equation}
\label{e:ZS}
\begin{cases}
x_1\in\Argmin{x\in\HH_1}{\boldsymbol{f}(x,x_2)
+\boldsymbol{\mathcal {L}}(x,x_2)}\\
x_2\in\Argmin{x\in\HH_2}{\boldsymbol{f}(x_1,x)
-\boldsymbol{\mathcal{L}}(x_1,x)}.
\end{cases}
\end{equation}
\end{example}

\begin{proposition}
\label{p:selle}
In Example~\ref{prob:selle}, suppose that there exists 
$(z_1,z_2)\in\HH_1\oplus\HH_2$ such that
\begin{equation}
\label{e:QCselle}
\big(\!-\!\nabli{1}\boldsymbol{\mathcal{L}}(z_1,z_2),
\nabli{2}\boldsymbol{\mathcal{L}}(z_1,z_2)\big)
\in\partial\boldsymbol{f}(z_1,z_2).
\end{equation}
Let $\varepsilon\in\left]0,1/(\chi+1)\right[$ and let 
$(\gamma_n)_{n\in\NN}$ be a sequence in 
$\left[\varepsilon,(1-\varepsilon)/\chi\right]$.
Moreover, let $(x_{1,0},x_{2,0})\in\HH_1\oplus\HH_2$, let 
$(a_{1,n})_{n\in\NN}$, $(b_{1,n})_{n\in\NN}$, and
$(c_{1,n})_{n\in\NN}$ be absolutely summable sequences in
$\HH_1$, and let $(a_{2,n})_{n\in\NN}$, $(b_{2,n})_{n\in\NN}$, and 
$(c_{2,n})_{n\in\NN}$ be absolutely summable sequences in $\HH_2$. 
Now consider the following routine. 
\begin{align}
\label{e:tseng1selle}
(\forall n\in\NN)\quad
&\left\lfloor 
\begin{array}{l}
y_{1,n}=x_{1,n}-\gamma_n(\nabli{1}
\boldsymbol{\mathcal{L}}(x_{1,n},x_{2,n})
+a_{1,n})\\
y_{2,n}=x_{2,n}+\gamma_n(\nabli{2}
\boldsymbol{\mathcal{L}}(x_{1,n},x_{2,n})
+a_{2,n})\\
(p_{1,n},p_{2,n})=\prox_{\gamma_n
\boldsymbol{f}}(y_{1,n},y_{2,n})
+(b_{1,n},b_{2,n})\\
q_{1,n}=p_{1,n}-\gamma_n(\nabli{1}
\boldsymbol{\mathcal{L}}(p_{1,n},p_{2,n})
+c_{1,n})\\
q_{2,n}=p_{2,n}+\gamma_n(\nabli{2}
\boldsymbol{\mathcal{L}}(p_{1,n},p_{2,n})
+c_{2,n})\\
x_{1,n+1}=x_{1,n}-y_{1,n}+q_{1,n}\\
x_{2,n+1}=x_{2,n}-y_{2,n}+q_{2,n}.
\end{array}
\right.
\end{align}
Then there exists a solution 
$(\overline{x}_1,\overline{x}_1)$ to Example~\ref{prob:selle} such 
that $x_{1,n}\weakly\overline{x}_1$, $p_{1,n}\weakly\overline{x}_1$,
$x_{2,n}\weakly\overline{x}_2$, and $p_{2,n}\weakly\overline{x}_2$.
\end{proposition}
\begin{proof} Example~\ref{prob:selle} corresponds to the 
particular instance of Problem~\ref{prob:main} in which $m=2$, 
$\boldsymbol{g}_1=\boldsymbol{\mathcal{L}}$, 
and $\boldsymbol{g}_2=-\boldsymbol{\mathcal{L}}$. Indeed, 
it follows from \cite[Theorem~1]{Rock70} that the operator
\begin{equation}
(x_1,x_2)\mapsto\big(\nabli{1}
\boldsymbol{\mathcal{L}}(x_1,x_2),
-\nabli{2}\boldsymbol{\mathcal{L}}(x_1,x_2)\big)
\end{equation}
is monotone in $\HH_1\oplus\HH_2$ and hence \eqref{e:Bmon} holds. 
In addition, \eqref{e:QCselle} implies \eqref{e:exist1} and, 
since $\nabla\boldsymbol{\mathcal{L}}$ is $\chi$--Lipschitzian, 
\eqref{e:BLip} holds.
Altogether, since \eqref{e:tseng1} reduces to \eqref{e:tseng1selle},
the result follows from Theorem~\ref{t:main}. 
\end{proof}

Next, we examine an application of Proposition~\ref{p:selle} to 
$2$-player finite zero-sum games.

\begin{example}
We consider a $2$-player finite zero-sum game 
(for complements and background on finite games, 
see~\cite{Weib95}). Let $S_1$ be the finite set of pure strategies 
of player $1$, with cardinality $N_1$, and let 
\begin{equation}
\label{e:C_i}
C_1=\Menge{(\xi_j)_{1\leq j\leq N_1}\in\left[0,1\right]^{N_1}}
{\sum_{j=1}^{N_1}\xi_j=1}
\end{equation}
be his set of mixed strategies ($S_2$, $N_2$, and $C_2$ are defined
likewise). Moreover, let $L$ be an $N_1\times N_2$ real cost matrix 
such that
\begin{equation}
\label{e:cqzs}
(\exi z_1\in C_1)(\exi z_2\in C_2)\quad
-Lz_2\in N_{C_1}z_1\:\:\:\text{and}\:\:\:L^{\top}z_1\in N_{C_2}z_2.
\end{equation}
The problem is to
\begin{equation}
\label{e:ZStrong2}
\text{find}\quad x_1\in\RR^{N_1}\:\:\:\text{and}\:\:\:x_2\in\RR^{N_2}
\quad\text{such that}\quad
\begin{cases}
x_1\in\Argmin{x\in C_1}{x^{\top}L x_2}\\
x_2\in\Argmax{x\in C_2}{x_1^{\top}L x}.
\end{cases}
\end{equation}
Since the penalty function of player 1 is 
$(x_1,x_2)\mapsto x_1^{\top}L x_2$ and the penalty function of player
2 is $(x_1,x_2)\mapsto -x_1^{\top}L x_2$, \eqref{e:ZStrong2} is a 
zero-sum game. It corresponds to the particular instance of 
Example~\ref{prob:selle} in which $\HH_1=\RR^{N_1}$, 
$\HH_2=\RR^{N_2}$, $\boldsymbol{f}\colon(x_1,x_2)\mapsto
\iota_{C_1}(x_1)+\iota_{C_2}(x_2)$, and $\boldsymbol{\mathcal{L}}
\colon(x_1,x_2)\mapsto x_1^{\top}Lx_2$. Indeed,
since $C_1$ and $C_2$ are nonempty closed convex sets, 
$\boldsymbol{f}\in\Gamma_0(\HH_1\oplus\HH_2)$. Moreover,
$x_1\mapsto\boldsymbol{\mathcal{L}}(x_1,x_2)$ and $x_2\mapsto-
\boldsymbol{\mathcal{L}}(x_1,x_2)$ are convex, and
$\nabla\boldsymbol{\mathcal{L}}\colon(x_1,x_2)
\mapsto(Lx_2,L^{\top}x_1)$ is linear and bounded, with
$\|\nabla\boldsymbol{\mathcal{L}}\|=\|L\|$. In addition, 
for every $\gamma\in\RPP$, $\prox_{\gamma
\boldsymbol{f}}=(P_{C_1},P_{C_2})$ \cite[Proposition~23.30]{Livre1}. 
Hence, \eqref{e:tseng1selle} reduces to (we set the error terms 
to zero for simplicity)
\begin{align}
\label{e:tseng1selle2}
(\forall n\in\NN)\quad
&\left\lfloor 
\begin{array}{l}
y_{1,n}=x_{1,n}-\gamma_nLx_{2,n}\\
y_{2,n}=x_{2,n}+\gamma_nL^{\top}x_{1,n}\\
p_{1,n}=P_{C_1}y_{1,n}\\
p_{2,n}=P_{C_2}y_{2,n}\\
q_{1,n}=p_{1,n}-\gamma_nLp_{2,n}\\
q_{2,n}=p_{2,n}+\gamma_nL^{\top}p_{1,n}\\
x_{1,n+1}=x_{1,n}-y_{1,n}+q_{1,n}\\
x_{2,n+1}=x_{2,n}-y_{2,n}+q_{2,n},
\end{array}
\right.
\end{align}
where $(\gamma_n)_{n\in\NN}$ is a sequence in 
$\left[\varepsilon,(1-\varepsilon)/\|L\|\right]$ for some
arbitrary $\varepsilon\in\left]0,1/(\|L\|+1)\right[$.
Since $\partial\boldsymbol{f}\colon(x_1,x_2)\mapsto
N_{C_1}x_1\times N_{C_2}x_2$, \eqref{e:cqzs} yields 
\eqref{e:QCselle}. Altogether, Proposition~\ref{p:selle} asserts 
that the sequences $(x_{1,n})_{n\in\NN}$ and $(x_{2,n})_{n\in\NN}$ 
generated by \eqref{e:tseng1selle2} converge to 
$\overline{x}_1\in\RR^{N_1}$ and
$\overline{x}_2\in\RR^{N_2}$, respectively, such that 
$(\overline{x}_1,\overline{x}_2)$ is a solution to 
\eqref{e:ZStrong2}.
\end{example}

\subsection{Generalized Nash equilibria}
We consider the particular case of Problem~\ref{prob:main} 
in which $\boldsymbol{f}$ is the indicator function of a closed convex 
subset of $\HHH=\HH_1\oplus\cdots\oplus\HH_m$. 
\begin{example}
\label{prob:gnep}
Let $\boldsymbol{C}\subset\HHH$ be a nonempty closed convex set and,
for every $i\in\{1,\ldots,m\}$, let 
$\boldsymbol{g}_i\colon\HHH\to\RX$
be a function which is differentiable with respect to its $i$th 
variable. Suppose that
\begin{multline}
\label{e:BmonGNEP}
\big(\forall (x_1,\ldots,x_m)\in\HHH\big)
\big(\forall (y_1,\ldots,y_m)\in\HHH\big)\\
\sum_{i=1}^m\scal{\nabli{i}\boldsymbol{g}_i
(x_1,\ldots,x_m)-\nabli{i}\boldsymbol{g}_i
(y_1,\ldots,y_m)}{x_i-y_i}\geq0
\end{multline} 
and set
\begin{equation}
(\forall (x_1,\ldots,x_m)\in\HHH)\quad
\begin{cases}
\hspace{0.4cm}\boldsymbol{Q}_1(x_2,\ldots,x_{m})\hspace{0.1cm}
=\hskip .2mm \menge{x\in\HH_1}{(x,x_2,\ldots,x_m)\in\boldsymbol{C}}\\
\hspace{3.25cm}\vdots\\
\boldsymbol{Q}_m(x_1,\ldots,x_{m-1})=\menge{x\in\HH_m}
{(x_1,\ldots,x_{m-1},x)\in\boldsymbol{C}}.
\end{cases} 
\end{equation}
The problem is to find $x_1\in\HH_1,\ldots,$ 
$x_m\in\HH_m$ such that
\begin{equation}
\begin{cases}
\,x_1\hskip-.3cm&\in\:\:\Argmin{x\in\boldsymbol{Q}_1(x_2,\ldots,
x_{m})}{\boldsymbol{g}_1(x,x_2,\ldots,x_m)}\\
&\:\vdots\\
x_m\hskip-.3cm&\in\:\:\Argmin{x\in\boldsymbol{Q}_m(x_1,\ldots,
x_{m-1})}{\boldsymbol{g}_m(x_1,\ldots, x_{m-1},x)}.
\end{cases}
\end{equation}
\end{example}

The solutions to Example~\ref{prob:gnep} are called generalized Nash 
equilibria \cite{Facc10}, social equilibria \cite{Debr52}, or 
equilibria of abstract economies \cite{Arro54}, and their existence 
has been studied in \cite{Arro54,Debr52}. We deduce from 
Proposition~\ref{p:aux} that we can find a solution to 
Example~\ref{prob:gnep} by solving a variational inequality in 
$\HHH$, provided the latter has solutions. This observation is also
made in \cite{Facc10}, which investigates a Euclidean setting in 
which additional smoothness properties are imposed on 
$(\boldsymbol{g}_i)_{1\leq i\leq m}$. 
An alternative approach for solving Example~\ref{prob:gnep} in 
Euclidean spaces is also proposed in \cite{VHeu09} with stronger 
differentiability properties on 
$(\boldsymbol{g}_i)_{1\leq i\leq m}$ and 
a monotonicity assumption of the form \eqref{e:BmonGNEP}. However, 
the convergence of the method is not guaranteed. Below we derive from
Section~\ref{ssec:1} a weakly convergent method for solving 
Example~\ref{prob:gnep}.

\begin{proposition}
\label{p:gnep}
In Example~\ref{prob:gnep}, suppose that there exist 
$(z_1,\ldots,z_m)\in\HHH$ such that
\begin{equation}
\label{e:existGNEP}
-\big(\nabli{1}\boldsymbol{g}_1
(z_1,\ldots,z_m),\ldots,\nabli{m}\boldsymbol{g}_m
(z_1,\ldots,z_m)\big)\in N_{\boldsymbol{C}}
(z_1,\ldots,z_m)
\end{equation}
and $\chi\in\RPP$ such that
\begin{multline}
\label{e:BLipGNEP}
(\forall (x_1,\ldots,x_m)\in\HHH)
(\forall (y_1,\ldots,y_m)\in\HHH)\\
\sum_{i=1}^m\|\nabli{i}\boldsymbol{g}_i
(x_1,\ldots,x_m)-\nabli{i}\boldsymbol{g}_i
(y_1,\ldots,y_m)\|^2\leq\chi^2\sum_{i=1}^m\|x_i-y_i\|^2.
\end{multline}
Let $\varepsilon\in\left]0,1/(\chi+1)\right[$ and let 
$(\gamma_n)_{n\in\NN}$ be a sequence in 
$\left[\varepsilon,(1-\varepsilon)/\chi\right]$.
Moreover, for every $i\in\{1,\ldots,m\}$, let $x_{i,0}\in\HH_i$,
and let $(a_{i,n})_{n\in\NN}$, $(b_{i,n})_{n\in\NN}$, and 
$(c_{i,n})_{n\in\NN}$ be absolutely summable sequences in $\HH_i$.
Now consider the following routine. 
\begin{align}
\label{e:tseng1gnep}
(\forall n\in\NN)\quad
&\left\lfloor 
\begin{array}{l}
\text{For}\:\:i=1,\ldots,m\\
\lfloor\:y_{i,n}=x_{i,n}-\gamma_n(\nabli{i}\boldsymbol{g}_i
(x_{1,n},\ldots,x_{m,n})+a_{i,n})\\
(p_{1,n},\ldots,p_{m,n})=P_{\boldsymbol{C}}(y_{1,n},\ldots,y_{m,n})
+(b_{1,n},\ldots,b_{m,n})\\
\text{For}\:\:i=1,\ldots,m\\
\left\lfloor
\begin{array}{l}
q_{i,n}=p_{i,n}-\gamma_n(\nabli{i}\boldsymbol{g}_i
(p_{1,n},\ldots,p_{m,n})+c_{i,n})\\
x_{i,n+1}=x_{i,n}-y_{i,n}+q_{i,n}.
\end{array}
\right.
\end{array}
\right.
\end{align}
Then there exists a solution 
$(\overline{x}_1,\ldots,\overline{x}_m)$ to Example~\ref{prob:gnep}
such that, for every $i\in\{1,\ldots,m\}$,
$x_{i,n}\weakly\overline{x}_i$ and $p_{i,n}\weakly\overline{x}_i$.
\end{proposition}
\begin{proof}
Example~\ref{prob:gnep} corresponds to the particular instance of 
Problem~\ref{prob:main} in which 
$\boldsymbol{f}=\iota_{\boldsymbol{C}}$.
Hence, since $P_{\boldsymbol{C}}=\prox_{\boldsymbol{f}}$,
the result follows from Theorem~\ref{t:main}.
\end{proof}

\subsection{Cyclic proximation problem}

We consider the following problem in 
$\HHH=\HH_1\oplus\cdots\oplus\HH_m$.

\begin{example}
\label{prob:cyclic} 
Let $\GG$ be a real Hilbert space, let 
$\boldsymbol{f}\in\Gamma_0(\HHH)$, and, for 
every $i\in\{1,\ldots,m\}$, let $L_i\colon\HH_i\to\GG$ be a bounded 
linear operator. The problem is to find $x_1\in\HH_1,\ldots,
x_m\in\HH_m$ such that
\begin{equation}
\label{e:cyclic}
\begin{cases}
\,x_1\hskip-.3cm&\in\Argmin{x\in\HH_1}{\boldsymbol{f}
(x,x_2,\ldots,x_m)+\displaystyle{\frac12\|L_1x-L_2x_2\|^2}}\\
\,x_2\hskip-.3cm&\in\Argmin{x\in\HH_2}{\boldsymbol{f}
(x_1,x,\ldots,x_m)+\displaystyle{\frac12\|L_2x-L_3x_3\|^2}}\\
&\,\vdots\\
x_m\hskip-.3cm&\in\Argmin{x\in\HH_m}{\boldsymbol{f}
(x_1,\ldots,x_{m-1},x)+\displaystyle{\frac12\|L_mx-L_1x_1\|^2}}.
\end{cases}
\end{equation}
\end{example}

For every $i\in\{1,\ldots,m\}$, the individual penalty 
function of player $i$ models his desire to
keep some linear transformation $L_i$ of his strategy close
to some linear transformation of that of the next player $i+1$.
In the particular case when $\boldsymbol{f}\colon(x_i)
_{1\leq i\leq m}\mapsto\sum_{i=1}^mf_i(x_i)$, 
a similar formulation is studied in \cite[Section~3.1]{Atto08},
where an algorithm is proposed for solving \eqref{e:cyclic}.
However, each step of the algorithm involves the proximity operator
of a sum of convex functions, which is extremely difficult to 
implement numerically. The method described below circumvents this
difficulty.
 
\begin{proposition}
\label{p:cyclic}
In Example~\ref{prob:cyclic}, suppose that there exists
$(z_1,\ldots,z_m)\in\HHH$ such that
\begin{equation}
\label{e:QCprox}
\big(L_1^*(L_2z_2-L_1z_1),\ldots,L_m^*(L_1z_1-L_mz_m)\big)\in
\partial\boldsymbol{f}(z_1,\ldots,z_m).
\end{equation}
Set $\chi=2\max_{1\leq i\leq m}\|L_i\|^2$, let 
$\varepsilon\in\left]0,2/(\chi+1)\right[$ and let 
$(\gamma_n)_{n\in\NN}$ be a sequence in 
$[\varepsilon,(2-\varepsilon)/\chi]$. For every 
$i\in\{1,\ldots,m\}$, let 
$x_{i,0}\in\HH_i$, and let $(a_{i,n})_{n\in\NN}$ and 
$(b_{i,n})_{n\in\NN}$ be absolutely summable sequences in $\HH_i$. 
Now set $L_{m+1}=L_1$, for every $n\in\NN$, set $x_{m+1,n}=x_{1,n}$, 
and consider the following routine.
\begin{align}
\label{e:FBcyclic}
(\forall n\in\NN)\quad
&\left\lfloor 
\begin{array}{l}
\text{For}\:\:i=1,\ldots,m\\
\lfloor\:y_{i,n}=x_{i,n}-\gamma_n\big(L_i^*(L_ix_{i,n}-
L_{i+1}x_{i+1,n})+a_{i,n}\big)\\[2mm]
(x_{1,n+1},\ldots,x_{m,n+1})=\prox_{\gamma_n\boldsymbol{f}}
(y_{1,n},\ldots,y_{m,n})+(b_{1,n},\ldots,b_{m,n}).
\end{array}
\right.
\end{align}
Then there exists a solution 
$(\overline{x}_1,\ldots,\overline{x}_m)$ to 
Example~\ref{prob:cyclic} such that, for every $i\in\{1,\ldots,m\}$,
$x_{i,n}\weakly\overline{x}_i$ and  
$L_i^*\big(L_i(x_{i,n}-\overline{x}_i)-L_{i+1}(x_{i+1,n}-
\overline{x}_{i+1})\big)\to 0$.
\end{proposition}
\begin{proof}
Note that Example~\ref{prob:cyclic} corresponds to the particular
instance of Problem~\ref{prob:main} in which, for every 
$i\in\{1,\ldots,m\}$, $\boldsymbol{g}_i\colon(x_i)_{1\leq i\leq m}
\mapsto\|L_ix_i-L_{i+1}x_{i+1}\|/2$, where we set
$x_{m+1}=x_1$. Indeed, since
\begin{equation}
\label{e:Bgrad}
(\forall (x_1,\ldots,x_m)\in\HHH)\quad
\begin{cases}
\:\:\nabli{1}\boldsymbol{g}_1(x_1,\ldots,x_m)\hskip-.3cm
&=L_1^*(L_1x_1-L_2x_2)\\
\hskip-.3cm&\,\,\vdots\\
\nabli{m}\boldsymbol{g}_m(x_1,\ldots,x_m)\hskip-.3cm
&=L_m^*(L_mx_m-L_1x_1),
\end{cases}
\end{equation} 
the operator $(x_i)_{1\leq i\leq m}\mapsto(\nabli{i}\boldsymbol{g}_i
(x_1,\ldots,x_m))_{1\leq i\leq m}$ is linear and bounded.
Thus, for every $(x_1,\ldots,x_m)\in\HHH$,
\begin{align}
\label{e:calcul}
\sum_{i=1}^m\scal{\nabli{i}\boldsymbol{g}_i(x_1,\ldots,x_m)}{x_i}
&=\sum_{i=1}^m\scal{L_i^*(L_ix_i-L_{i+1}x_{i+1})}{x_i}\nonumber\\
&=\sum_{i=1}^m\scal{L_ix_i-L_{i+1}x_{i+1}}{L_ix_i}\nonumber\\
&=\sum_{i=1}^m\|L_ix_i\|^2-\sum_{i=1}^m
\scal{L_{i+1}x_{i+1}}{L_ix_i}\nonumber\\
&=\frac12\sum_{i=1}^m\|L_ix_i\|^2+\frac12
\sum_{i=1}^m\|L_{i+1}x_{i+1}\|^2
-\sum_{i=1}^m\scal{L_{i+1}x_{i+1}}{L_ix_i}\nonumber\\
&=\sum_{i=1}^m\frac12\|L_ix_i-L_{i+1}x_{i+1}\|^2\nonumber\\
&=\sum_{i=1}^m\frac{1}{2\|L_i\|^2}\|L_i\|^2
\|L_ix_i-L_{i+1}x_{i+1}\|^2\nonumber\\
&\geq\chi^{-1}\sum_{i=1}^m\|L_i^*(L_ix_i-L_{i+1}x_{i+1})\|^2
\nonumber\\
&=\chi^{-1}\sum_{i=1}^m\|\nabli{i}\boldsymbol{g}_i(x_1,
\ldots,x_m)\|^2,
\end{align}
and hence \eqref{e:Bcoco} and \eqref{e:Bmon} hold. In addition,
\eqref{e:QCprox} yields \eqref{e:exist2}.
Altogether, since 
\eqref{e:FB} reduces to \eqref{e:FBcyclic}, the 
result follows from Theorem~\ref{t:main0}.
\end{proof}

We present below an application of Proposition~\ref{p:cyclic}
to cyclic proximation problems and, in particular,
to cyclic projection problems.

\begin{example}
We apply Example~\ref{prob:cyclic} to cyclic evaluations of 
proximity operators. For every $i\in\{1,\ldots,m\}$, let 
$\HH_i=\HH$, let $f_i\in\Gamma_0(\HH)$, let $L_i=\Id$,
and set $\boldsymbol{f}\colon(x_i)_{1\leq i\leq m}
\mapsto\sum_{i=1}^mf_i(x_i)$. In view of \eqref{e:prox}, 
Example~\ref{prob:cyclic} reduces to finding
$x_1\in\HH,\ldots,x_m\in\HH$ such that
\begin{equation}
\label{e:cycle}
\begin{cases}
x_1=\prox_{f_1}x_2\\
x_2=\prox_{f_2}x_3\\
\hspace{0.7cm}\vdots\\
x_m=\prox_{f_m}x_1.
\end{cases} 
\end{equation}
It is assumed that \eqref{e:cycle} has at least one solution. 
Since $\prox_{\boldsymbol{f}}\colon(x_i)_{1\leq i
\leq m}\mapsto(\prox_{f_i}x_i)_{1\leq i\leq m}$
\cite[Proposition~23.30]{Livre1}, 
\eqref{e:FBcyclic} becomes (we set errors to zero for simplicity)
\begin{equation}
\label{e:FBcyclicP}
(\forall n\in\NN)\quad
\left\lfloor 
\begin{array}{l}
\text{For}\:\:i=1,\ldots,m\\
\lfloor\:x_{i,n+1}=\prox_{\gamma_nf_i}((1-\gamma_n)x_{i,n}+\gamma_nx_{i+1,n}), 
\end{array}
\right.
\end{equation}
where $(x_{i,0})_{1\leq i\leq m}\in\HH^m$ and $(\gamma_n)_{n\in\NN}$ 
is a sequence in $\left[\varepsilon,1-\varepsilon\right]$ for some
arbitrary $\varepsilon\in\left]0,1/2\right[$. Proposition~\ref{p:cyclic} 
asserts that the sequences $(x_{1,n})_{n\in\NN},\ldots,
(x_{m,n})_{n\in\NN}$ generated by \eqref{e:FBcyclicP} converge weakly to
points $\overline{x}_1\in\HH,\ldots,\overline{x}_m\in\HH$, respectively, 
such that $(\overline{x}_1,\ldots,\overline{x}_m)$ is a solution to 
\eqref{e:cycle}.

In the particular case when, for every $i\in\{1,\ldots,m\}$, 
$f_i=\iota_{C_i}$, a solution of \eqref{e:cycle} represents a cycle 
of points in $C_1,\ldots,C_m$. It can be interpreted as a Nash 
equilibrium of the game in which, for every $i\in\{1,\ldots,m\}$, 
the strategies of player $i$, belong to $C_i$ and its penalty 
function is $(x_i)_{1\leq i\leq m}\mapsto\|x_i-x_{i+1}\|^2$, 
that is, player $i$ wants to have strategies as close as possible 
to the strategies of player $i+1$. Such schemes go back at least to
\cite{Gubi67}. It has recently been proved \cite{Bacc11} that, in 
this case, if $m>2$, the cycles are not minimizers 
of any potential, from which we infer that this problem cannot 
be reduced to a potential game. Note that \eqref{e:FBcyclicP} 
becomes
\begin{equation}
\label{e:FBcyclicP2}
(\forall n\in\NN)\quad
\left\lfloor 
\begin{array}{l}
\text{For}\:\:i=1,\ldots,m\\
\lfloor\:x_{i,n+1}=P_{C_i}((1-\gamma_n)x_{i,n}+\gamma_nx_{i+1,n})
\end{array}
\right.
\end{equation}
and the sequences thus generated $(x_{1,n})_{n\in\NN},\ldots,
(x_{m,n})_{n\in\NN}$ converge weakly to
points $\overline{x}_1\in\HH,\ldots,\overline{x}_m\in\HH$, respectively, 
such that $(\overline{x}_1,\ldots,\overline{x}_m)$ is a cycle. 
The existence of cycles
has been proved in \cite{Gubi67} when one of the sets 
$C_1,\ldots,C_m$ is bounded. Thus, \eqref{e:FBcyclicP2} is an 
alternative parallel algorithm to the method of successive 
projections \cite{Gubi67}.
\end{example}

\end{document}